\documentclass[12pt]{article}
\usepackage{amsmath}
\usepackage{amsthm}
\newtheorem{definition}{Definition}[section]
\newtheorem{example}{Example}[section]
\newtheorem{theorem}{Theorem}[section]
\usepackage{enumerate}
\usepackage[english]{babel}
\usepackage[utf8x]{inputenc}

\usepackage{lmodern,amsmath,amssymb,latexsym}
\title{Non Existence of Manifold Structures on Topological constructions}
\author{Mohammad Al Attar }
\date{March 2022}

\begin{document}

\maketitle

Abstract: We show the non existence of manifold structure on certain topological constructions. The main results of this paper are:  (1) The wedge of manifold with itself is not a manifold. (2) The cone of a manifold is not a manifold.  The main tools for the main results are homological algebra and algebraic topology.

\newpage

Topological manifolds are spaces that resemble euclidean space locally. That is, at a sufficiently small region, the manifold looks like (a region of) $\mathbb{R}^n$. For instance, a line is a one dimensional manifold. A surface is a two dimensional manifold and so on. A piece of paper is a real world model of a two dimensional manifold. The earth, locally, looks like a 2 dimensional region, and is in some sense locally flat. However, to discuss the notion of flatness from the point of view of curvature, one requires additional structure on a topological manifold. In particular, a topological manifold is not equipped with data the takes into account curvature. However, manifolds that take into account curvature are called $\textbf{Riemannian manifolds}$, which are, informally, manifolds, such that the notions of differentiation and integration are well defined;  in other words are \textbf{smooth manifolds}, with the additional structure that you can define the notions of length on a linear approximation of the manifold at a given region of any point. It is well known that there are only two 'reasonable'  1-dimensional manifolds from a topological sense: The real line and the circle. It is clear that these two spaces are topologically distinct. For 2-manifolds, the situation gets a bit more complicated. In particular 'reasonable' 2-dimensional manifolds are characterised by the number of holes they have. The number of holes is called the genus. For instance, the two sphere has no holes and therefore has genus 0. However, one may obtain a genus one surface by attaching a handle to the two sphere; obtaining a torus. This process is an inductive process, yielding an orientable manifold at each stage. However, one may also obtain non orientable manifolds by attaching the boundary of the mobius band. In higher dimensions, the situation gets much more complicated and is an active area of research. In this paper, we mainly discuss topological manifolds of arbitrary dimensions and show that certain topological constructions do not yield a topological manifold structure. The background assumes a solid background in algebraic topology. Hatcher's book [1] and [3] are very good sources for algebraic topology.  $[2]$ is a nice introduction to topological manifolds. $[4]$ Is the common reference for point set topology. The structure of this paper is as follows. Chapter 2 discusses some backgrounds needed for and used in chapter 3. The third chapter contains the main results. The author has embarked on the goal to write complete proofs of the results discussed in that section.

\section{Background}
In this section, we discuss some of the background necessary for the second section. In pursuit of understanding the second section better, we define a few key notions. We illustrate some of the notions defined with a few examples. The examples in this section are standard and can be found in $[1]$ and $[2]$ along with proofs. Our notation follows as that in $[1]$ and $[3]$

\begin{definition}
A topological manifold of dimension $n$ is a topological space $M$ that is second countable, Hausdorff and locally euclidean of dimension $n$: For each $p\in M$ there exists a pair $(U,\phi)$, where $U$ is an open subset of $M$ containing $p$ and $\phi:U\rightarrow \mathbb{R}^n$ is a homeomorphism. The pair $(U,\phi)$ is called a $\textbf{chart}$ on $M$.  $U$ is called a $\textbf{coordinate domain}$ , and $\phi$ is called a $\textbf{coordinate map}$. The collection of charts is called an $\textbf{atlas}$ on $M$.  Letting $r^1,...,r^n$ denote the standard coordinates on $\mathbb{R}^n$, we have $\phi = (x^1,...,x^n)$ , where $x^i=r^i\circ \phi$. Thus, one calls the  data $(U,x^1,...,x^n)$ or more simply $(x^1,...,x^n)$, a $\textbf{local coordinate system}$ on $M$.

\end{definition}

For a point $p$ in a manifold $M$, one typically identifies $p$ with its image under a coordinate map. To see this in action, let $(U,x^1...,x^n)$ be a chart on $M$ . Then one may blur the distinction between $p$ and $(x^1(p),...,x^n(p))$.  The reason for such a temporary identification is as follows: Given a point $p\in M$, a chart about $p$ identifies, atleast topologically, a neighborhood of $p$ with an open region of $\mathbb{R}^n$.  Thus, we would like the coordinate system on the region surrounding the point to be induced from the standard coordinates on euclidean space. This is indeed the case, under our identification. More precisely, identifying $U$ with its image under the coordinate map, we obtain $x^i=r^i|_{U}$. Note that this identification is not canonical, since obviously different charts may give rise to different coordinates of the point.

\begin{example}
$\mathbb{R}^n$ for each $n$ with the standard topology is a manifold of dimension $n$. For each $n$, $\mathbb{S}^n$ is a compact topological manifold of dimension $n$. $\mathbb{R}^n$ has a global coordinate system. Namely, $(\mathbb{R}^n,r^1,...,r^n)$.  $\mathbb{S}^n$ cannot have a global coordinate system, for otherwise we would have a homeomorphism $\mathbb{R}^n\approx \mathbb{S}^n$.

\end{example}

\begin{definition}
Let $(C_{*},\partial)$ be a chain complex of abelian groups. Then, $\partial^2=0$. We thus define for each n $\textbf{n-th homology group}$ of $C_{*}$: $H_n(C_{*})=\frac{im\partial_n}{ker\partial_{n+1}}$.

\end{definition}

\begin{example}
Let $X$ be a space and $A\subseteq X$ a subspace. Then we have a pair $(X,A)$. For each $n$, $S_n(A)$ is the free abelian group generated by all the singular n simplices in $A$. Since an singular n simplex in $A$ is an n simplex in $X$, $S_n(A)\subseteq S_n(X)$. We have a short exact sequence of chain complexes: $0\rightarrow S_{*}(A)\hookrightarrow S_{*}(X)\rightarrow S_{*}(X,A)\rightarrow 0$.Where we define $S_n(X,A)$ to be the group $S_{n}(X)$ modulo $S_{n}(A)$. We note that we may realize $S_n(X,A)$ as the free abelian group generated by the $n$ simplices in $X$ with image not contained in $A$.

\end{example}

\begin{theorem} (Excision)
Let $X$ be a space and $Z\subseteq A\subseteq X$  with $\overline{Z}\subseteq intA$. Then the inclusion $(X-Z,A-Z)\hookrightarrow (X,A)$ induces an isomorphism $H_n(X-Z,A-Z)\cong H_n(X,A)$  for each $n$.

\end{theorem}

\begin{example}
Suppose $X$ is a $T_1$ space and $x\in X$. Then for any open neighborhood $U$ of $x$,

\begin{enumerate}
\item $H_n(X,U)\cong  H_n(X-\{x\},U-\{x\})$.
\item $H_n(X,X-\{x\})\cong H_n(U,U-\{x\})$.
\end{enumerate}

\noindent First one is a direct consequence of excision. Let us establish the second one: Note that $\{x\}\subseteq U$ and thus $X\backslash U \subseteq X\backslash \{x\}$. Therefore, $H_n(X\backslash X\backslash U, X\backslash \{x\} \backslash X\backslash U)=H_n(U,U\backslash \{x\})\cong H_n(X,X\backslash \{x\})$.

\end{example}

\begin{definition}
If $X$ is a space we define the $\textbf{cone } of $ $X$, denoted $CX$, to be the quotient space $\frac{X\times [0,1]}{X\times 1}$. If $X$ and $Y$ are spaces then choosing $x\in X$ and $y\in Y$, we define the $\textbf{wedge sum}$ of $X$ and $Y$ to be the quotient space $\frac{X\coprod Y}{x\sim y}$.

\end{definition}

\begin{example}
If $X$ and $Y$ are homeomorphic spaces, then $CX$ is homeomorphic to $CY$; $CX\approx CY$. The reason for this is the functoriality of the cone map: That is, if $f:X\rightarrow Y$ is continuous, then the $\textbf{cone map}$ is given by $Cf:CX\rightarrow CY$, $(Cf)([x,t]=[f(x),t]$. The cone map respects composition and preserves the identity . The respecting of said properties allows us to deduce that homeomorphic spaces have homeomorphic cones.

\end{example}

\begin{theorem} (Relative Mayer-Vietoris)
Suppose we have a pair of space $(X,Y)=(int A \cup int B, int C\cup intD)$ , where $C\subseteq A$ and $D\subseteq B$. Then, we have the following long exact sequence in homology:\\

\centerline{$...\rightarrow H_n(A\cap B,C\cap D)\rightarrow H_n(A,C)\oplus H_n(B,D)\rightarrow H_n(X,Y)\rightarrow... $}

\end{theorem}

\section{Main Results}

The proof technique of the next result follows from trying to get control of the situation by using the Relative Mayer Vietoris sequence from algebraic topology. The reason for doing so is to reduce problems about manifolds to problems about local homology groups $H_{*}(X,X\backslash \{x\})$ for $x\in X$.

\begin{theorem}
Suppose $M$ is a manifold of dimension $n\geq 2$. Then, $M\vee M$ is not a topological manifold.

\end{theorem}

\begin{proof}
We first make the claim that if $M$ is an $n$ dimensional manifold and $p_1,...,p_m\in M$.  Then\\

\centerline{$H_n(M,M\backslash \{p_1,...,p_m\})\cong \mathbb{Z}^m$}

\noindent\\Indeed, it suffices by inducting on $m$ to establish the $m=2$ case. To that end, choose open sets $U_1$, $U_2$ about $p_1$ and $p_2$ respectively for which $U_1\cap U_2=\varnothing$. By excision, $H_n(M,M\backslash \{p_1,p_2\}) $ $\cong H_n(U_1\cup U_2, U_1\backslash \{p_1 \}\cup U_2\backslash \{p_2 \})\cong H_n(U_1,U_1-p_1)\oplus H_n(U_2,U_2-p_2)\cong \mathbb{Z}\oplus \mathbb{Z}$. The second to last isomorphism follows from the relative version of Mayer Vietoris. Now to prove the theorem, suppose $M\vee M$ is a topological manifold of dimension $n$. Then,\\

\centerline{ $H_n(M\vee M,M\vee M\backslash \{\star\})\cong \mathbb{Z}$.}
\noindent \\ Let $p_1$ and $p_2$ denote the base points. Now we claim:\\

\centerline{$H_n(M\vee M,(M\vee M)\backslash \{\star\})\cong \mathbb{Z}\oplus \mathbb{Z}$}

\noindent \\ To prove the result above, we use the Relative Mayer Vietoris sequence. Clearly it suffices to show:\\

\centerline{$H_n(M\vee M, (M\vee M)\backslash \{\star\})\cong H_n(M,M\backslash \{p_1\})\oplus H_n(M,M\backslash \{p_1\})$}

\noindent \\ To show the above, we invoke the Relative Mayer-Vietoris sequence, note that $M\vee M=M_1\cup M_2$, where we have homeomorphisms $M_1\approx M$ and $M_2\approx M$ that restrict to  homeomorphisms $M_i\backslash \{\star\} \approx M \backslash \{p_i\}$ for each $i$. Thus consider $(M\vee M, M\vee M\backslash \{\star\})=(M_1\cup M_2, M_1\backslash \{\star\})\cup M_2\backslash\{\star\})$. Hence we have by Relative Mayer-Vietoris\\

\centerline{$...\rightarrow H_n(\star)\rightarrow H_n(M_1,M_1\backslash \{\star\})\oplus H_n(M_2,M_2\backslash \{\star\})\rightarrow H_n(M\vee M, M\vee M\backslash \{\star\})\rightarrow H_{n-1}(\star)\rightarrow...$}

\noindent \\Since  $n\geq 2$, $H_n(\star)=H_{n-1}(\star)=0$ . Hence the long exact sequence above gives the isomorphism we seek.

\end{proof}

The motivation behind the second result is the case when when we are dealing with $\mathbb{S}^1$. We have $C\mathbb{S}^1=D^2$ . The cone of a point $x$, $Cx=I$ . None of these cones admit a manifold structure (without boundary!). Thus, one can naturally ask if $CM$ is ever a manifold when $M$ is a manifold. The intuition of the proof technique is to get the discussion closer to a discussion of local homology groups. This is because, as is known (see [1] and the sixth example in the preceding section), that the dimension of a manifold is characterized by its local homology groups $H_n(X,X\backslash x)$ for each $x\in X$. Thus, to get the closer discussion to local homology groups, one requires to go through a degree of technical details. Here is another motivating example

\begin{theorem} (Motivating) If $M$ is a compact 2-manifold then $CM$ is not a 3-manifold.

\end{theorem}

\begin{proof}
Suppose for a contradiction that $CM$ is a 3 manifold. Since $M$ is compact, so is its cone. Since the cone of any space is contractible, it must be simply connected. Therefore by the topological Poincare Conjecture, $CM\approx \mathbb{S}^3$. However, $H_3(\mathbb{S}^3)=\mathbb{Z}$, whereas since cones are homotopy equivalent to a point, we have $H_3(CM)=0$.

\end{proof}

\noindent Squeezing more out of the proof above, we obtain the following example:

\begin{example}
$C\mathbb{R}P^2$ is not a topological manifold of dimension 3.

\end{example}

\begin{theorem}
If $M$ is a manifold then $CM$ is not a topological manifold.

\end{theorem}

\begin{proof}
We recall that $CM$ is obtained by forming a closed cylinder of finite height, via attaching a copy of $M$ to each $t$, where $t\in [0,1]$ then identifying the points of the in the copy of $M$ at $t=1$; obtaining a "vertex" point. Let $\star$ denote the vertex point. If $CM$ is a manifold, then it must be a manifold of dimension $n+1$. Hence, $CM\backslash \{\star\}$ is also a manifold of dimension $n+1$. Now note that $CM\backslash \{ \star\}  = M\times [0,1)$. Since $0$ is a deformation retract of  $[0,1)$,  $\tilde{H}_m(CM\backslash \{\star\})\cong\tilde{H}_m(M)$ for each $m$.  Now  note that if $p\in M\times 0$ , then $p=(x,0)$. Thus, \\

\centerline{$\tilde{H}_m(CM\backslash \{\star\},CM\backslash \{ \star, p\}) \cong \tilde{H}_m(M\times [0,1), (M\times [0,1))\backslash \{p\})$ }

\noindent \\ To establish the result we seek, it suffices to establish the following two claims.

\noindent \\ \textit{Claim 1:} $\tilde{H}_m(M\times [0,1), (M\times [0,1))\backslash \{p\})\cong \tilde{H_m}(M\times 0 , (M\times 0)\backslash \{p\})$\\

To prove the claim it suffices to show that the pair  $(M\times [0,1),(M\times [0,1)\backslash \{p\}))$ is homotopy equivalent to $(M\times 0 , (M\times 0)\backslash \{p\})$. To see this, first note that $0$ is a deformation retract of $[0,1)$ via the map\\

\centerline{$H:[0,1)\times I \rightarrow [0,1)$, given by $H(s,t)=(1-t)s$}

\noindent \\Now $H$ induces a continuous map $H^{M}: (M\times [0,1))\times I\rightarrow M\times [0,1)$ given by:\\

\centerline{ $H^{M}((m,s),t)=(m,H(s,t))$}

\noindent \\ It is clear that $H^{M}$ deformation retracts $M\times [0,1)$ onto  $M\times 0$.  Now, the retraction associated with the homotopy above is the map\\

\centerline{$r:M\times [0,1)\rightarrow  M\times 0$ , given by $r(x,s)=(x,0)$}

\noindent \\ Thus, $r((M\times [0,1)) \backslash \{p\})=(M\times 0)\backslash \{x\}$ . As a consequence, we get that $H$ induces a homotopy equivalence $(M\times [0,1),M \times [0,1)\backslash \{p\}) \simeq (M\times 0, (M\times 0)\backslash \{p\})$. \\

\noindent \textit{Claim 2:} $\tilde{H_m}(M\times 0 , (M\times 0)\backslash \{p\})\cong \tilde{H}_m(M,M\backslash \{x\})$ 

 \noindent \\ This follows because we have a canonical homeomorphism $M\rightarrow M\times \{0\}$ which induces a homeomorphism of pairs $(M, M\backslash \{x\})\approx (M\times 0, (M\times 0)\backslash \{p\})$.

\noindent \\ However, since the top local homology groups must differ for dimensional reasons, we obtain a contradiction.

\end{proof}

\section{Conclusion}

In this paper, we first provided a proof of showing that wedging a manifold with itself (of dimension atleast 2) is not a manifold. Then, we have written a proof to show that the cone of a manifold is not a topological manifold The result can be extended in various directions. For instance, $C\mathbb{S}^1$ is a topological manifold with boundary. Thus, one can investigate the conditions conditions for which the cone of a manifold (as a manifold with boundary) admits a smooth structure. Moreover, one can seek to find an analogous result to the results presented in this paper using cohomology.

\section{References}

\indent [1] Pi.math.cornell.edu. 2022. Allen Hatcher's Homepage. [online] Available at: <https://pi.math.cornell.edu/~hatcher/>

\noindent  [2]  Lee, J., 2011. Introduction to topological manifolds. New York: Springer.

\noindent  [3] Rotman, J., 2013. An Introduction to Algebraic Topology. New York, NY: Springer.

\noindent  [4] Munkres, J., 2013. Topology. Harlow: Pearson.

\end{document}